\documentclass{amsart}

\usepackage{preamble}
\usepackage{macros}

\title{Norms and Hermitian $\mathrm{K}$-Theory}
\author{Brian Shin}
\date{\today}

\begin{document}

\begin{abstract}
    Over the past century, cohomology operations have played a crucial role in homotopy theory and its applications.
    A powerful framework for constructing such operations is the theory of commutative algebras in spectra.
    In this article, we discuss an algebro-geometric analogue of this framework, called the theory of \emph{normed algebras} in motivic spectra.
    Specifically, we show that the motivic spectrum $\mathrm{ko}$ representing very effective hermitian $\mathrm{K}$-theory can be equipped with a normed algebra structure, and that the orientation map $\mathrm{MSL} \to \mathrm{ko}$ respects this structure.
    The main step will be showing that the motivic infinite loop space machine is compatible with norms.
\end{abstract}

\maketitle

\setcounter{tocdepth}{1}
\tableofcontents

\section{Introduction} \label{section:intro-intro}

Multiplicative cohomology theories occupy a central role in classical homotopy theory.
In favorable circumstances, the spectra that represent multiplicative cohomology theories come with a highly structured multiplication known as a commutative (\emph{i.e.}, $\mathbb{E}_\infty$) algebra structure.
At the level of cohomology theories, this structure manifests in the form of power operations.
For example, the spectrum $\mathrm{H}\mathbb{F}_p$ representing ordinary cohomology with $\mathbb{F}_p$-coefficients admits a commutative algebra structure, and the corresponding power operations can be recognized as the Steenrod operations.
These operations have been fundamental to the study and application of homotopy theory since the very beginnings of the subject.

In motivic homotopy theory, the role played by commutative algebras in spectra is taken up by \emph{normed algebras} in motivic spectra as defined by Bachmann--Hoyois in \cite{BachmannHoyois_Norms}.
These are motivic spectra equipped with multiplications that are ``more commutative than commutative'' in a sense to be made precise later in this article. (See Section~\ref{section:normed-stuff}.)
Normed algebras support a theory of power operations which is significantly richer than the analogous theory coming from commutative algebras.
In particular, the resulting theory is rich enough to recover Voevodsky's motivic Steenrod operations (see \cite[Example 7.25]{BachmannHoyois_Norms}).
These power operations were crucial in the resolution (by Voevodsky, Rost, and others) of the Milnor and Bloch--Kato conjectures relating motivic cohomology to étale cohomology.

The goal of the present article is to discuss the following theorem.
Here, we write $\mathrm{ko}_S$ for the motivic spectrum representing \emph{very effective hermitian $\mathrm{K}$-theory}.
This is the motivic analogue of the ordinary spectrum $\mathrm{ko}$ representing connective real topological $\mathrm{K}$-theory.
In particular, it is an example of a cohomology theory which \emph{does not} admit a theory of Chern classes, and can therefore be used to detect interesting phenomena in motivic homotopy theory.

\begin{theorem}[\cite{Shin_NormsTransfers}]\label{theorem:main}
    Let $S$ be a quasi-compact quasi-separated scheme.
    The motivic spectrum $\mathrm{ko}_S \in \SH(S)$ admits a normed algebra structure.
    In fact, the diagram
    \begin{equation*}
        \begin{tikzcd}
            \mathrm{MSL}_S \ar[d] \ar[r] & \mathrm{ko}_S \ar[d] \\
            \mathrm{MGL}_S \ar[r] & \mathrm{kgl}_S
        \end{tikzcd}
    \end{equation*}
    can be refined to one of normed algebras in $\SH(S)$.
\end{theorem}

\subsection*{Outline}

We begin in Section~\ref{section:intr-to-homotopy} with a rapid overview of motivic homotopy theory.
To motivate the theory of normed algebras, we take a brief detour in Section~\ref{section:commutative-stuff} to reframe the theory of commutative (\emph{i.e.}, $\mathbb{E}_\infty$) algebras.
Following \cite{BachmannHoyois_Norms}, we dive into the theory of normed algebras in Section~\ref{section:normed-stuff}; there, we indicate in particular the main difficulty (Tension~\ref{tension:no-easy-norms}) in proving Theorem~\ref{theorem:main}.
Our primary tool in overcoming Tension~\ref{tension:no-easy-norms} is the theory of motivic infinite loop spaces from \cite{EHKSY_deloop1}, which we review in the Section~\ref{section:loops}.
Section~\ref{section:end-game} is dedicated to proving Theorem~\ref{theorem:main}.

\subsection*{Conventions}
Throughout, all schemes are assumed quasi-compact and quasi-separated over $\mathbb{Z}$, and all smooth morphisms are assumed finitely presented.
Given a scheme $S$, we'll write $\Sm_S$ for the category of smooth $S$-schemes.
We will use the framework of $\infty$-categories.
We will refer to $\infty$-groupoids as anima, and write $\Ani$ for the $\infty$-category of anima.

\subsection*{Acknowledgements}
We thank the organizers of the 2024 session of the IAS/Park City Mathematics Institute for putting together an incredible experience.
We are particularly grateful to Marc Levine for inviting us to give a talk at that session and contribute this article.
We extend a hearty thanks to Tom Bachmann, Elden Elmanto, Jeremiah Heller, Marc Hoyois, and Denis Nardin for the invaluable conversations related to the work discussed here.
We also thank the anonymous referee, whose comments have greatly improved the exposition.

\section{A Brief Tour of Motivic Homotopy Theory} \label{section:intr-to-homotopy}

\begin{notation}
    Let $S$ be a scheme, which is fixed for the rest of the article.
\end{notation}

In short, the idea behind motivic homotopy theory is that it is homotopy theory where smooth manifolds are replaced by smooth $S$-schemes.
The role played by the real line $\mathbb{R}$ is then taken up by the affine line $\mathbb{A}^1_S$.

\begin{definition} \label{definition:motivic-spaces}
    A \emph{motivic space} over $S$ is a Nisnevich sheaf
    \begin{equation*}
        F : \Sm_S^\op \to \Ani
    \end{equation*}
    such that $F(X) \xrightarrow{\simeq} F(X \times_S \mathbb{A}^1_S)$ for all $X \in \Sm_S$.
    We'll write $\HH(S) \subseteq \PSh(\Sm_S)$ for the full subcategory spanned by motivic spaces over $S$.
\end{definition}

The inclusion $\HH(S) \subseteq \PSh(\Sm_S)$ admits an accessible left adjoint $\mathrm{L}_\mathrm{mot}$ called \emph{motivic localization}.
In particular, the $\infty$-category $\HH(S)$ is presentable, and any presheaf $F : \Sm_S^\op \to \Ani$, be it constant or representable, gives rise to a motivic space $\mathrm{L}_\mathrm{mot} F \in \HH(S)$.
We will drop $\mathrm{L}_\mathrm{mot}$ from the notation when it is clear from context.

\begin{remark}
    If one takes the above definition and replaces $\Sm_S$ with the category of smooth manifolds, the Nisnevich topology with the open cover topology, and $\mathbb{A}^1_S$ with $\mathbb{R}$, then one recovers the usual $\infty$-category $\Ani$ of anima.
\end{remark}

\begin{remark}
    The Nisnevich topology sits between the Zariski and the étale topologies, which may be more familiar to the algebro-geometrically inclined reader.
    In principle, one could define a version of motivic spaces for any topology one wishes.
    However, the Nisnevich topology stands out for several reasons.
    For one, it is the coarsest topology for which Morel--Voevodsky's localization theorem \cite[Section~3,~Theorem~2.21]{MV_A1} (also known as gluing) holds.
    This is a vast generalization of the exact sequence
    \begin{equation*}
        \mathrm{CH}_n(Z) \xrightarrow{i_*} \mathrm{CH}_n(X) \to \mathrm{CH}_n(X \smallsetminus Z) \to 0
    \end{equation*}
    of Chow groups, where $i : Z \to X$ is a closed immersion of varieties.
    The localization theorem is a foundational result of Morel--Voevodsky's theory and forms one of the key ingredients of the six functor formalism in this setting.
    See \cite{Ayoub_6ff}, \cite{CisinskiDeglise_TriangulatedMixedMotives}, and \cite{Hoyois_6ff} for more on this.

    Another reason to use the Nisnevich topology is that it is the finest topology for which algebraic $\mathrm{K}$-theory satisfies descent.
    We stress in particular that the étale topology is often \emph{too fine} for the purposes of motivic homotopy theory.
    This point will come back to haunt us later (see Tension~\ref{tension:no-easy-norms}).
\end{remark}

The $\infty$-category $\HH(S)_*$ of \emph{pointed} motivic spaces is also presentable.
Better yet, we have a symmetric monoidal structure on $\HH(S)_*$ called the \emph{smash product} that is uniquely characterized by the fact that (1) it preserves colimits in each variable, and (2) $X_+ \otimes Y_+ \cong (X \times_S Y)_+$ for all $X, Y \in \HH(S)$.

\begin{exercise}
    Prove that there is an equivalence $S^1 \otimes (\mathbb{A}^1_S \smallsetminus 0) \cong \mathbb{P}^1_S$ in $\HH(S)_*$.
\end{exercise}

\begin{definition}
    For $X \in \mathscr{H}(S)_*$, we define $\Omega_\mathbb{P} X = \Hom_*(\mathbb{P}^1_S, X)$.
    This is the \emph{$\mathbb{P}^1$-loop space} of $X$.
\end{definition}

\begin{definition}
    A \emph{motivic spectrum} $E$ over $S$ consists of
    \begin{itemize}
        \item a sequence $E_0, E_1, E_2, \ldots \in \HH(S)_*$ of pointed motivic spaces, together with
        \item a sequence $E_i \xrightarrow{\simeq} \Omega_\mathbb{P} E_{i+1}, i \in \mathbb{N},$ of equivalences.
    \end{itemize}
    We refer to $E_0$ as the \emph{motivic infinite loop space} of $E$, and denote it by $\Omega^\infty_\mathbb{P} E$.
    Motivic spectra over $S$ can be organized into an $\infty$-category denoted $\SH(S)$.
\end{definition}

The functor $\Omega^\infty_\mathbb{P} : \SH(S) \to \HH(S)_*$ admits a left adjoint $\Sigma^\infty_\mathbb{P} : \HH(S)_* \to \SH(S)$.
In fact, the functor $\Sigma^\infty_\mathbb{P}$ exhibits $\SH(S)$ as the presentably symmetric monoidal $\infty$-category obtained from $\HH(S)_*$ by inverting $\mathbb{P}^1_S$ under smash product.
See \cite[Section~2.1]{Robalo_KTheoryBridge} and \cite[Section~6.1]{Hoyois_6ff} for more details on inverting objects in presentably symmetric monoidal $\infty$-categories.
Note that the equivalence $\mathbb{P}^1_S \cong S^1 \otimes (\mathbb{A}^1_S \smallsetminus 0)$ in $\HH(S)_*$ implies that $\SH(S)$ is stable.

One of the original motivations for developing motivic homotopy theory was to create a framework in which the algebraic $\mathrm{K}$-theory of smooth varieties could be represented by an algebro-geometric version of a homotopy-theoretic spectrum.
Indeed, this is accomplished as follows.

\begin{example}
    For $X \in \Sch$, we write $\Vect(X)$ for the anima of vector bundles on $X$.
    This is a commutative monoid under the operation of direct sum.
    We define $\mathrm{K} : \Sch^\op \to \Ani_*$ to be the Zariski sheafification of the presheaf $X \mapsto \Vect(X)^\mathrm{gp}$, where $({-})^\mathrm{gp}$ denotes the group completion functor.
    See \cite{GGN} for more on group completion.
    The anima $\mathrm{K}(X)$ is called the \emph{algebraic $\mathrm{K}$-theory anima} of $X$.

    We'll define now a motivic spectrum $\mathrm{KGL}_S \in \SH(S)$ that represents (the motivic localization of) the algebraic $\mathrm{K}$-theory of smooth $S$-schemes.
    To simplify the exposition, we'll focus on the case when $S$ is regular. (See Remark~\ref{remark:irregular} for the general case.)
    Let $\mathrm{K}_S : \Sm_S^\op \to \Ani_*$ to be the restriction of $\mathrm{K}$ along $\Sm_S \to \Sch$.
    Regularity of $S$ implies that $\mathrm{K}_S$ is a (pointed) motivic space over $S$.
    The tautological line bundle $\mathcal{O}(1)$ on $\mathbb{P}^1_S$ gives a canonical element $\beta \in (\Omega_\mathbb{P} \mathrm{K}_S)(S)$ called the \emph{Bott element}.
    Using the projective bundle formula, one finds that multiplication by $\beta$ induces an equivalence $\mathrm{K}_S \xrightarrow{\simeq} \Omega_\mathbb{P} \mathrm{K}_S$.
    Thus, we can define a periodic motivic spectrum
    \begin{equation*}
        \mathrm{KGL}_S = (\mathrm{K}_S, \mathrm{K}_S, \ldots) \in \SH(S).
    \end{equation*}
    This is the \emph{algebraic $\mathrm{K}$-theory spectrum}.
\end{example}

An important variant of algebraic $\mathrm{K}$-theory is \emph{Grothendieck--Witt theory}, also known as \emph{hermitian $\mathrm{K}$-theory}.
This is a version of algebraic $\mathrm{K}$-theory in which one studies vector bundles equipped with a non-degenerate symmetric bilinear form.
We refer to \cite{Hornbostel_Representability}, \cite{Schlichting_HermitianK}, \cite{HJNY_KO}, and \cite{CHN} for more Grothendieck--Witt theory.
Our exposition here follows closely that of \cite[Section~7]{HJNY_KO}.

\begin{example}
    For $X \in \Sch$, we write $\Vect^\sym(X)$ for the anima of vector bundles on $X$ equipped with a non-degenerate symmetric bilinear form.
    This is a commutative monoid under the operation of orthogonal direct sum.
    We define $\mathrm{GW}^s : \Sch^\op \to \Ani_*$ to be the Zariski sheafification of the presheaf $X \mapsto \Vect^\sym(X)^\mathrm{gp}$.
    The anima $\mathrm{GW}^s(X)$ is called the \emph{symmetric Grothendieck--Witt theory anima} of $X$.

    We'll define now a motivic spectrum $\mathrm{KO}_S \in \SH(S)$ that represents (the motivic localization of) the symmetric Grothendieck--Witt theory of smooth $S$-schemes.
    Here, we focus on the case when $S$ is regular with $2 \in \mathcal{O}(S)^\times$.
    (See Remark~\ref{remark:irregular} and Remark~\ref{remark:too} for the general case.)
    Let $\mathrm{GW}^s_S : \Sm_S^\op \to \Ani_*$ be the restriction of $\mathrm{GW}^s$ along $\Sm_S \to \Sch$.
    Using regularity of $S$, one can show that $\mathrm{GW}^s_S$ is a (pointed) motivic space over $S$.
    There is a canonical element $\tilde{\beta} \in (\Omega^4_\mathbb{P}\mathrm{GW}^s_S)(S)$ called the \emph{hermitian Bott element}.
    Multiplication by $\tilde{\beta}$ induces an equivalence $\mathrm{GW}^s_S \xrightarrow{\simeq} \Omega^4_\mathbb{P} \mathrm{GW}^s_S$.
    Thus, we can define a 4-periodic motivic spectrum
    \begin{equation*}
        \mathrm{KO}_S = (\mathrm{L}_\mot \mathrm{GW}^s_S, \Omega^3_\mathbb{P} \mathrm{L}_\mot \mathrm{GW}^s_S, \ldots, \mathrm{L}_\mot \mathrm{GW}^s_S, \Omega^3_\mathbb{P} \mathrm{L}_\mot \mathrm{GW}^s_S, \ldots) \in \SH(S)
    \end{equation*}
    where every fourth space is $\mathrm{L}_\mot \mathrm{GW}^s_S$.
    This is the \emph{hermitian $\mathrm{K}$-theory spectrum}.
\end{example}

\begin{remark} \label{remark:irregular}
    When $S$ is not assumed regular, the presheaves $\mathrm{K}_S$ and $\mathrm{GW}_S^s$ can fail to be motivic spaces.
    Nevertheless, one can define motivic spectra $\mathrm{KGL}_S$ and $\mathrm{KO}_S$ by applying $\mathrm{L}_\mot$ everywhere and making appropriate modifications.
    See \cite[Section~7]{HJNY_KO} for details in the case of $\mathrm{KO}_S$.
    For emphasis, the resulting motivic spectra are often referred to as the \emph{homotopy} algebraic $\mathrm{K}$-theory and the \emph{homotopy} hermitian $\mathrm{K}$-theory spectra.
\end{remark}

\begin{remark} \label{remark:too}
    In \cite{Hornbostel_Representability}, \cite{Schlichting_HermitianK}, and \cite{HJNY_KO}, one finds a pervasive assumption that $2$ is invertible on the base scheme.
    Indeed, the theory of symmetric forms is far more intricate when $2$ is not invertible.
    In recent work \cite{Nine_One} of Calmès--Dotto--Harpaz--Hebestreit--Land--Moi--Nardin--Nikolaus--Steimle, the picture without this assumption has become much sharper.
    In particular, we are now able to drop this assumption entirely.
    See \cite[Remark~7.3]{HJNY_KO} and \cite{CHN}.
\end{remark}

\begin{remark} \label{remark:some-objects-of-SH}
    One can view the motivic spectra $\mathrm{KGL}_S$ and $\mathrm{KO}_S$ as being motivic analogues of the topological $\mathrm{K}$-theory spectra $\mathrm{KU}$ and $\mathrm{KO}$ from ordinary homotopy theory.
    In fact, many spectra from ordinary homotopy theory admit analogues in the motivic world.\footnote{
        The precise relationship between these motivic spectra and their ``ordinary'' counterparts is that, when $S$ is the spectrum of a subfield of $\mathbb{C}$, there is a functor $r : \SH(S) \to \Spt$ known as \emph{complex Betti realization} which send the former to the latter.
    }
    The table below contains a short list of examples with ordinary spectra on the left and their motivic analogues on the right.

    \begin{table}[h]
        \centering
        \rowcolors{2}{gray!5}{gray!15}
        \begin{tabular}{c l c l}
            $\Spt$ & & $\SH(S)$ \\
            $\mathrm{H}\mathbb{Z}$ & ordinary cohomology & $\mathrm{H}\mathbb{Z}_S$ & motivic cohomology \\
            $\mathrm{KU}$ & complex topological $\mathrm{K}$-theory & $\mathrm{KGL}_S$ & algebraic $\mathrm{K}$-theory \\
            $\mathrm{KO}$ & real topological $\mathrm{K}$-theory & $\mathrm{KO}_S$ & hermitian $\mathrm{K}$-theory \\
            $\mathrm{MU}$ & complex cobordism & $\mathrm{MGL}_S$ & algebraic cobordism \\
            $\mathrm{MSU}$ & oriented complex cobordism & $\mathrm{MSL}_S$ & oriented algebraic cobordism
        \end{tabular}
    \end{table}

    We refer the reader to \cite{BEM} for more on motivic cohomology,\footnote{In particular, we are taking $\mathrm{H}\mathbb{Z}_S$ to mean the zero slice of $\mathrm{KGL}_S$ in this article.} and \cite[Section~16]{BachmannHoyois_Norms} for more on the cobordism spectra.
\end{remark}

With an eye towards infinite loop space recognition, we introduce the following, which is a motivic analogue of connectivity for ordinary spectra.

\begin{definition}
    A motivic spectrum over $S$ is \emph{very effective} if it is in the full subcategory $\SH(S)^\mathrm{veff} \subseteq \SH(S)$ generated under colimits and extensions by $\Sigma^\infty_\mathbb{P} X_+$ with $X \in \Sm_S$.\footnote{
        There are in fact several notions in motivic homotopy theory that could be thought of as analogues of connectivity.
        Notable examples include \emph{homotopy connectivity} and \emph{effectivity}.
        See \cite{Bachmann_GeneralizedSlices} and \cite[Appendix~B]{BachmannHoyois_Norms} for more on these.
        Each notion has its uses, but very effectiveness is the only one that will play a role in this article.
    }
\end{definition}

By \cite[Proposition 1.4.4.11]{Lurie_HA}, the inclusion $\SH(S)^\mathrm{veff} \subseteq \SH(S)$ admits a right adjoint which sends a motivic spectrum to its \emph{very effective cover}.
We write $\mathrm{kgl}_S$ and $\mathrm{ko}_S$ for the very effective covers of $\mathrm{KGL}_S$ and $\mathrm{KO}_S$.
The motivic spectra $\mathrm{H}\mathbb{Z}_S$, $\mathrm{MGL}_S$, and $\mathrm{MSL}_S$ are already very effective.

Since $\SH(S)$ is stable and presentably symmetric monoidal, it provides an excellent setting for doing higher algebra.
For example, all of the motivic spectra introduced so far are actually commutative algebras in $\SH(S)$, and consequently the generalized cohomology theories they represent are multiplicative.

\begin{remark}\label{remark:SH-is-CMon}
    Note that the assignment $S \mapsto \SH(S)$ can actually be made into a functor $\SH : \Sch^\op \to \CMon(\Cat_\infty)$.
    If $f : S \to T$ is a morphism of schemes, the pullback functor $f^* : \SH(T) \to \SH(S)$ is characterized by the fact that (1) it is symmetric monoidal, (2) it preserves all colimits, and (3) we have $f^* (\Sigma^\infty_\mathbb{P} Y_+) \cong \Sigma^\infty_\mathbb{P} (Y_S)_+$ for all $Y \in \Sm_T$.
    Moreover, we have $\SH(S_1 \sqcup S_2) \cong \SH(S_1) \times \SH(S_2)$ for all $S_1, S_2 \in \Sch$.
\end{remark}

\section{Commutative Monoids and Commutative Algebras via Correspondences} \label{section:commutative-stuff}

As a warmup to the theory of norm monoidal structures, let us meditate on a more basic situation.
Let $F : \Sm_S^\op \to \Ani$ be a presheaf.
What does it mean to give $F$ the structure of a commutative monoid?
More precisely, we contemplate the following:

\begin{question}
    What does it mean to lift $F$ along the forgetful functor $\CMon(\Ani) \to \Ani$?
\end{question}

Such a lift consists of a lot of data.
For example, we'd have, for every $X \in \Sm_S$ and every natural number $n$, a map of anima
\begin{equation*}
    \mu_n : F(X)^{\times n} \to F(X)
\end{equation*}
given by $n$-ary multiplication.
The rest of the data would encode coherences such as unitality, associativity, and commutativity of the multiplication.

If $F$ preserves finite products (for instance, we could have that $F$ is a sheaf for the Zariski topology) then we can rewrite the $n$-ary multiplication as
\begin{equation*}
    F\bigl(X^{\sqcup n}\bigr) \cong F(X)^{\times n} \xrightarrow{\mu_n} F(X)
\end{equation*}
and we can view it as an extra functoriality of the presheaf $F$.
Namely, we can view it as a kind of pushforward along the \emph{fold} (also known as \emph{codiagonal}) map $\nabla : X^{\sqcup n} \to X$.

We'll see soon that commutative monoid structures on $F$ can be identified with this sort of extra functoriality.
To encode this extra functoriality, we are led to consider \emph{correspondences} (also known as \emph{spans}) in $\Sm_S$.

\begin{definition}
    Let $\fold$ denote the class of morphisms in $\Sm_S$ which can be written as a finite coproduct of fold maps $X^{\sqcup n} \to X$.
    We write $\Corr^{\fold}(\Sm_S) = \mathscr{S}\mathrm{pan}(\Sm_S, \fold, \mathrm{all})$ for the $\infty$-category of \emph{$\fold$-correspondences} in $\Sm_S$.
    The objects of $\Corr^{\fold}(\Sm_S)$ are the same as those of $\Sm_S$, and morphisms in $\Corr^{\fold}(\Sm_S)$ from $X$ to $Y$ are \emph{correspondences} (also known as \emph{spans})
    \begin{equation*}
        \begin{tikzcd}[row sep=small]
            & \widetilde{X} \ar[dl,"p"'] \ar[dr,"f"] \\
            X & & Y
        \end{tikzcd}
    \end{equation*}
    in $\Sm_S$ where $p$ is in $\fold$.
    Composition of correspondences is determined by taking pullbacks.
\end{definition}

We refer to \cite[Section 5]{Barwick_Mackey} for a construction of $\Corr^{\fold}(\Sm_S)$ as a complete Segal space.
The notation there is $A^\textit{eff}(\Sm_S, \fold, \mathrm{all})$ and it is referred to as the \emph{effect Burnside category}.

\begin{remark}
    Note that every morphism in $\Corr^{\fold}(\Sm_S)$ admits a factorization
    \begin{equation*}
        \left(\begin{tikzcd}[row sep=small]
            & \widetilde{X} \ar[dl,"p"'] \ar[dr,"f"] \\
            X & & Y
        \end{tikzcd}\right) = \left(
            \begin{tikzcd}[row sep=small]
            & \widetilde{X} \ar[dl,"p"'] \ar[dr,"\mathrm{id}"] & & \widetilde{X} \ar[dl,"\mathrm{id}"'] \ar[dr,"f"] \\
            X & & \widetilde{X} & & Y
        \end{tikzcd}
        \right)
    \end{equation*}
    into a ``backwards morphism'' followed by a ``forwards'' morphism.
    In fact, this type of factorization is essentially unique.
    Consequently, we can think of functors $F : \Corr^{\fold}(\Sm_S)^\op \to \Ani$ as consisting of
    \begin{itemize}
        \item for every $X \in \Sm_S$, an anima $F(X) \in \Ani$,
        \item for every morphism $f : X \to Y$ in $\Sm_S$, a pullback morphism $f^* : F(Y) \to F(X)$,
        \item for every morphism $p : \widetilde{X} \to X$ in $\fold$, a pushforward morphism $p_\otimes : F(\widetilde{X}) \to F(X)$.
    \end{itemize}
    This comes also with data witnessing the coherence of these maps.

    In summary, we may view presheaves on $\Corr^{\fold}(\Sm_S)$ as presheaves on $\Sm_S$ equipped with some extra covariant functoriality along maps in $\fold$.
    We can view $\Sm_S$ as the wide subcategory of $\Corr^{\fold}(\Sm_S)$ consisting of ``forwards'' morphisms.
    Restriction along the inclusion $\Sm_S \to \Corr^{\fold}(\Sm_S)$ forgets about the extra covariant functoriality.
\end{remark}

\begin{exercise}
    Show that the functor $\Sm_S \to \Corr^{\fold}(\Sm_S)$ preserves finite coproducts.
    Conclude that a presheaf on $\Corr^{\fold}(\Sm_S)$ preserves finite products if and only if its restriction to $\Sm_S$ does.
\end{exercise}

We now formalize the claim that, under very mild conditions, commutative monoid structures can be encoded as extra functoriality.

\begin{proposition}[See {\cite[Proposition C.5]{BachmannHoyois_Norms}}]\label{prop:cmon-lifts-extensions}
    Let $D$ be an $\infty$-category with finite products, and let $F : \Sm_S^\op \to D$ be a functor that preserves finite products.
    There is an equivalence
    \begin{equation*}
        \left\{
            \begin{tikzcd}
                & \CMon(D) \ar[d] \\
                \Sm_S^\op \ar[r,"F"'] \ar[ur, dashed] & D
            \end{tikzcd}
        \right\} \cong \left\{
            \begin{tikzcd}
                \Sm_S^\op \ar[r,"F"] \ar[d] & D \\
                \Corr^{\fold}(\Sm_S)^\op \ar[ur,dashed]
            \end{tikzcd}
        \right\}
    \end{equation*}
    between lifts along the forgetful functor $\CMon(D) \to D$ and extensions along the inclusion $\Sm_S \to \Corr^{\fold}(\Sm_S)$.
\end{proposition}

In other words, in the context of presheaves on $\Sm_S$, we may identify algebraic structure as being extra functoriality.
This is a theme that will reoccur throughout this article.

With the above discussion in mind, we make the following:

\begin{definition}
    A \emph{symmetric monoidal $\infty$-category} over $S$ is a functor $\Corr^{\fold}(\Sm_S)^\op \to \Cat_\infty$ that preserves finite products.
\end{definition}

\begin{example} \label{example:SH-extended-to-fold}
    Using Proposition~\ref{prop:cmon-lifts-extensions} and Remark~\ref{remark:SH-is-CMon}, we can encode the smash product on $\SH(X)$ for $X \in \Sm_S$ in terms of an extension
    \begin{equation*}
        \begin{tikzcd}
            \Sm_S^\op \ar[r,"\SH"] \ar[d] & \Cat_\infty. \\
            \Corr^{\fold}(\Sm_S)^\op \ar[ur,dashed,"\SH"']
        \end{tikzcd}
    \end{equation*}
    The resulting functor is a symmetric monoidal $\infty$-category over $S$.
\end{example}

Given a symmetric monoidal $\infty$-category $T : \Corr^{\fold}(\Sm_S)^\op \to \Cat_\infty$ over $S$, we will usually use the notation
\begin{align*}
    X &\mapsto T(X)\\
    \Bigl(Y \xleftarrow{f} X \Bigr) &\mapsto \Bigl(T(Y) \xrightarrow{f^*} T(X) \Bigr) \\
    \Bigl(\widetilde{X} \xrightarrow{p} X \Bigr) &\mapsto \Bigl(T (\widetilde{X}) \xrightarrow{p_\otimes} T(X) \Bigr)
\end{align*}
for the action of $T$ on the objects, forwards morphisms, and backwards morphisms in $\Corr^{\fold}(\Sm_S)$, respectively.

\begin{exercise}
    Let $T : \Corr^{\fold}(\Sm_S)^\op \to \Cat_\infty$ be a symmetric monoidal $\infty$-category over $S$.
    Let $X$ be a smooth $S$-scheme, let $n \geq 0$ be a natural number, and consider the fold map $\nabla : X^{\sqcup n} \to X$.
    Prove that
    \begin{equation*}
        \nabla_\otimes \nabla^* (E) \cong E^{\otimes n}
    \end{equation*}
    for all $E \in T(X)$.
\end{exercise}

We can also recast commutative algebras in terms of $\Corr^{\fold}(\Sm_S)$.
The precise details are a bit technical, so we'll be impressionistic here.

\begin{definition}
        Let $T : \Corr^{\fold}(\Sm_S)^\op \to \Cat_\infty$ be a symmetric monoidal $\infty$-category over $S$.
        Roughly speaking, a \emph{section} $E$ of (the coCartesian fibration associated to) the functor $T : \Corr^{\fold}(\Sm_S)^\op \to \Cat_\infty$ consists of
    \begin{itemize}
        \item for every $X \in \Sm_S$, an object $E_X \in T(X)$,
        \item for every morphism $f : X \to Y$ in $\Sm_S$, a morphism $f^\sharp : f^* (E_Y) \to E_X$, and
        \item for every morphism $p : \widetilde{X} \to X$ in $\fold$, a morphism $\mu_p : p_\otimes (E_{\widetilde{X}}) \to E_X$.
    \end{itemize}
    There are also further data that encode an infinite hierarchy of coherence.
\end{definition}

\begin{proposition}\label{prop:CAlg-as-sections}
    Let $T : \Corr^{\fold}(\Sm_S)^\op \to \Cat_\infty$ be a symmetric monoidal $\infty$-category over $S$.
    There is an equivalence between
    \begin{enumerate}
        \item commutative algebras in $T(S)$, and
        \item sections $E$ of $T : \Corr^{\fold}(\Sm_S)^\op \to \Cat_\infty$ such that, for every map $f : X \to Y$ in $\Sm_S$, the morphism $f^\sharp : f^* (E_Y) \to E_X$ is an equivalence.
    \end{enumerate}
\end{proposition}

To illustrate the idea, suppose $E$ is a section satisfying Condition~(2) of Proposition~\ref{prop:CAlg-as-sections}.
If $p : \widetilde{X} \to X$ is the fold map $X^{\sqcup n} \to X$, then we have a morphism
\begin{equation*}
    (E_X)^{\otimes n} \cong p_\otimes p^* (E_X) \cong p_\otimes (E_{\widetilde{X}}) \xrightarrow{\mu_p} E_X.
\end{equation*}
These encode the $n$-ary multiplication maps for a commutative algebra structure on $E_X$.
In fact, since $S$ is a final object in $\Sm_S$, we see that each $E_X \in \CAlg(T(X))$ is the pullback of $E_S \in \CAlg(T(S))$ along the structure map $X \to S$.

\begin{remark}
    All of the above works with $\Sm_S$ replaced by an arbitrary \emph{extensive} $\infty$-category $C$.
    See \cite[Appendix~C]{BachmannHoyois_Norms}.
    Extensivity ensures that finite coproducts exist and are sufficiently nice. 
    This is necessary in order to define the $\infty$-category $\Corr^{\fold}(C)$.
\end{remark}

\section{Norms in Motivic Homotopy Theory} \label{section:normed-stuff}

As mentioned in Section~\ref{section:intro-intro}, the theory of normed algebras was developed in order to capture a sort of algebra in which the multiplication operation is ``more commutative than commutative''.
Having reframed the theory of commutative algebras in terms of pushforwards along fold maps, we can be a bit more precise about what we mean by this: allow for pushforwards along more maps.

\begin{definition}
    Let $\fet$ denote the class of finite étale morphisms in $\Sm_S$.
    We write $\Corr^{\fet}(\Sm_S) = \mathscr{S}\mathrm{pan}(\Sm_S, \fet,\mathrm{all})$ for the $\infty$-category of \emph{$\fet$-correspondences} in $\Sm_S$.
\end{definition}

\begin{definition}[See {\cite[Definition 6.5]{BachmannHoyois_Norms}}]
    A \emph{norm monoidal $\infty$-category} over $S$ is a functor $\Corr^\fet(\Sm_S)^\op \to \Cat_\infty$ that preserves finite products.
\end{definition}

Given a norm monoidal $\infty$-category $T : \Corr^\fet(\Sm_S)^\op \to \Cat_\infty$, we will usually use the notation
\begin{align*}
    X &\mapsto T(X)\\
    \Bigl(Y \xleftarrow{f} X \Bigr) &\mapsto \Bigl(T(Y) \xrightarrow{f^*} T(X) \Bigr) \\
    \Bigl(\widetilde{X} \xrightarrow{p} X \Bigr) &\mapsto \Bigl(T (\widetilde{X}) \xrightarrow{p_\otimes} T(X) \Bigr)
\end{align*}
for the action of $T$ on the objects, forwards morphisms, and backwards morphisms in $\Corr^\fet(\Sm_S)$, respectively.
We refer to the functor $p_\otimes$ as the \emph{norm} along the finite étale map $p$.

Note that fold maps are finite étale, so we can identify $\Corr^{\fold}(\Sm_S)$ as a wide subcategory of $\Corr^\fet(\Sm_S)$.
Consequently, normed monoidal structures are enhancements of symmetric monoidal structures.

\begin{remark}
    More generally, if $D$ is an $\infty$-category with finite products, then we may define a \emph{normed monoid} of objects in $D$ over $S$ to be a functor $\Corr^\fet(\Sm_S)^\op \to D$ that preserves finite products.
\end{remark}

\begin{remark}
    The term ``norm monoidal'' is inspired by Galois theory.
    Let $L/K$ be a finite Galois extension of fields with Galois group $G$.
    The \emph{norm} of an element $x \in L^\times$ is defined to be $\mathrm{Nm}(x) = \prod_{\sigma \in G} \sigma(x)$.
    Note that this product lives in $K^\times$ since it is fixed by the Galois action.

    The norm functors of a norm monoidal $\infty$-category should be thought of as a generalized and categorified version of this idea.
    Namely, if $p : \widetilde{X} \to X$ is a finite Galois cover (and in particular finite étale) then the norm functor $p_\otimes$ should be thought of as a ``tensor product indexed by the Galois group''.
\end{remark}

\begin{example}
    Consider the functor $\Sm_{({-})} : \Sm_S^\op \to \Cat_\infty$ with
    \begin{align*}
        X &\mapsto \Sm_X \\
        \Bigl(Y \xleftarrow{f} X \Bigr) &\mapsto \Bigl(\Sm_Y \xrightarrow{f^*} \Sm_X \Bigr).
    \end{align*}
    This extends to a norm monoidal $\infty$-category $\Sm_{({-})} : \Corr^\fet(\Sm_S)^\op \to \Cat_\infty$ where the norm functor along a finite étale map $p : \widetilde{X} \to X$ is given by the right adjoint $\mathrm{R}_p : \Sm_{\widetilde{X}} \to \Sm_X$ of the pullback functor $p^* : \Sm_X \to \Sm_{\widetilde{X}}$, also known as \emph{Weil restriction} along $p$.\footnote{
        Strictly speaking, the right adjoint of $p^* : \Sm_X \to \Sm_{\widetilde{X}}$ doesn't actually exist in this generality.
        To remedy the situation, one can either restrict to quasi-projective schemes, or extend the theory to allow for qcqs algebraic spaces.
        The first option is used in \cite{BachmannHoyois_Norms}.
        Since every qcqs algebraic space admits a Nisnevich cover by quasi-projectives schemes \cite[Chapter~II, Theorem~6.4]{Knutson_AlgSpc}, both options are harmless from the perspective of motivic homotopy theory and give the same end result.
    }
\end{example}

\begin{exercise}
    Let $L/K$ be a finite Galois extension of fields with Galois group $G$, and let $\Spec B$ be a smooth affine $L$-scheme.
    Consider the $L$-algebra
    \begin{equation*}
        \widetilde{A} = \bigotimes_{g \in G} g^* B
    \end{equation*}
    where $g^* B = L \otimes_{L, g} B$ is the base change of $B$ along $g : L \to L$, and the tensor product is over $L$.
    Note that $\widetilde{A}$ comes with a (semilinear!) action of $G$ given by permuting the tensor factors.
    Let $A = \widetilde{A}^G$ be the $k$-algebra of elements fixed by the action of $G$.
    Prove that $\Spec A$ is the Weil restriction of $\Spec B$ along $\Spec L \to \Spec K$.
\end{exercise}

\begin{exercise}
    Let $X$ be a smooth $S$-scheme.
    Verify that Weil restriction along $\nabla : X \sqcup X \to X$ is equivalent to the binary Cartesian product functor
    \begin{equation*}
        \Sm_X \times \Sm_X \to \Sm_X
    \end{equation*}
    under the identification $\Sm_{X \sqcup X} \cong \Sm_X \times \Sm_X$.
\end{exercise}

\begin{proposition}[See {\cite[Section 6.1]{BachmannHoyois_Norms}}] \label{prop:SH-is-normed}
    There is a norm monoidal $\infty$-category
    \begin{equation*}
        \begin{tikzcd}
            \Corr^{\fold}(\Sm_S)^\op \ar[d] \ar[r, "\SH"] & \Cat_\infty \\
            \Corr^{\fet}(\Sm_S)^\op \ar[ur, dashed, "\SH"']
        \end{tikzcd}
    \end{equation*}
    over $S$ which extends the smash product from Example~\ref{example:SH-extended-to-fold}.
\end{proposition}

The norm functor $p_\otimes : \SH(\widetilde{X}) \to \SH(X)$ for $p : \widetilde{X} \to X$ is essentially characterized by the requirement that
\begin{itemize}
    \item it is symmetric monoidal,
    \item it preserves sifted colimits, and
    \item for $X \in \Sm_S$, we have $p_\otimes (\Sigma^\infty_{\mathbb{P}} X_+) \cong \Sigma^\infty_\mathbb{P} \mathrm{R}_p(X)_+$.
\end{itemize}
To prove Proposition~\ref{prop:SH-is-normed}, the essential idea is to observe that the each step in the passage from $\Sm_X$ to $\SH(X)$ is suitably compatible with Weil restriction.

We also have a notion of algebra for norm monoidal structures.
Let $T : \Corr^\fet(\Sm_S)^\op \to \Cat_\infty$ be a norm monoidal $\infty$-category over $\Sm_S$.
Sections of $T : \Corr^\fet(\Sm_S)^\op \to \Cat_\infty$ are essentially the same as sections of a functor $\Corr^{\fold}(\Sm_S)^\op \to \Cat_\infty$ except that we have $\mu_p : p_\otimes E_{\widetilde{X}} \to E_X$ for all finite étale maps $p : \widetilde{X} \to X$.

\begin{definition}[See {\cite[Definition 7.1]{BachmannHoyois_Norms}}]
    Let $T : \Corr^\fet(\Sm_S)^\op \to \Cat_\infty$ be a norm monoidal $\infty$-category over $S$.
    A \emph{normed algebra} in $T$ over $S$ is a section $E$ of $T : \Corr^\fet(\Sm_S)^\op \to \Cat_\infty$ such that $f^\sharp : f^* E_Y \to E_X$ is an equivalence for every $f : X \to Y$ in $\Sm_S$.
    We'll write $\NAlg(T(S))$ for the $\infty$-category of normed algebras in $T$ over $S$.
\end{definition}

\begin{theorem}[\cite{BachmannHoyois_Norms}]
    The following motivic spectra admit lifts to $\NAlg(\SH(S))$.\footnote{We remind the reader that we are taking $\mathrm{H}\mathbb{Z}_S$ to be the zero slice of $\mathrm{KGL}_S$.
    See Remark~\ref{remark:some-objects-of-SH}.}
    \begin{enumerate}
        \item The motivic cohomology spectrum $\mathrm{H}\mathbb{Z}_S$.
        \item The algebraic $\mathrm{K}$-theory spectrum $\mathrm{KGL}_S$ and its very effective cover $\mathrm{kgl}_S$.
        \item The algebraic cobordism spectra $\mathrm{MGL}_S$ and $\mathrm{MSL}_S$.
    \end{enumerate}
\end{theorem}

\begin{remark} \label{remark:constructing-nalg-examples}
    Each of the normed algebra structures mentioned above requires a substantial amount of work to establish.
    For $\mathrm{KGL}_S$, this requires constructing norms for Robalo's noncommutative version of $\SH$.
    From here, the normed algebra structure on $\mathrm{H}\mathbb{Z}_S$ is the established by studying the compatibility of norms with the slice filtration.
    Finally, for the cobordism spectra, this requires a careful analysis of the multiplicative properties of the Thom construction.
\end{remark}

Generally, the construction of norm monoidal $\infty$-categories and normed algebras requires some nontrivial work.
In particular, it is often not at all obvious how to extend a given symmetric monoidal $\infty$-category over $S$ to a norm monoidal one, nor how to extend a commutative algebra to a normed algebra.
However, there is one type of situation in which such extensions are easy. 

\begin{slogan} \label{slogan:norms-by-descent}
    In the presence of finite étale descent, symmetric monoidal structures \emph{automatically} upgrade to norm monoidal structures.
\end{slogan}

More precisely, we have the following:

\begin{proposition}[See {\cite[Proposition C.11]{BachmannHoyois_Norms}}]
    Let $D$ be a presentable $\infty$-category, and let $F : \Sm_S^\op \to D$ be a presheaf that satisfies descent for the finite étale topology.
    Then there is an equivalence
    \begin{equation*}
        \left\{
            \begin{tikzcd}
                \Sm_S^\op \ar[r,"F"] \ar[d] & D \\
                \Corr^{\fold}(\Sm_S)^\op \ar[ur,dashed]
            \end{tikzcd}
        \right\} \cong \left\{
            \begin{tikzcd}
                \Sm_S^\op \ar[r,"F"] \ar[d] & D \\
                \Corr^{\fet}(\Sm_S)^\op \ar[ur,dashed]
            \end{tikzcd}
        \right\}
    \end{equation*}
    between symmetric monoidal structures on $F$ and norm monoidal structures on $F$.
\end{proposition}

The key idea is that finite étale maps are those maps which are (finite étale)-locally in $\fold$.
There is a similar statement for upgrading commutative algebras to normed algebras in the presence of finite étale descent, which is unfortunately a bit more technical to state.
See \cite[Corollary C.16]{BachmannHoyois_Norms} for the precise statement.

This leads to the basic tension for norm monoidal structures in motivic homotopy theory.

\begin{tension} \label{tension:no-easy-norms}
    For invariants which have finite étale descent, it is easy to upgrade a commutative algebra structure to a normed algebra structure.
    However, invariants coming from motivic homotopy theory usually do not have this kind of descent.
\end{tension}

We remark that a presheaf satisfies étale descent if and only if it satisfies both finite étale descent and Nisnevich descent.
See \cite[Theorem~B.6.4.1]{Lurie_SAG} for the analogous statement in the setting of spectral algebraic geometry.

One might hope that there are methods to resolve Tension~\ref{tension:no-easy-norms} in certain cases of interest.
As mentioned in Remark~\ref{remark:constructing-nalg-examples}, several such methods are developed already in \cite{BachmannHoyois_Norms}.
We'll see below that the theory of motivic infinite loop spaces provides another such method.
This will be our primary tool in proving Theorem~\ref{theorem:main}.

\section{Motivic Infinite Loop Space Theory} \label{section:loops}

The classical recognition principle for infinite loop spaces says the following:

\begin{theorem}[Recognition Principle, \cite{BoardmanVogt_EVERYTHING}, \cite{May_Iterated}, \cite{Segal_CatsCoh}]
    The following hold.
    \begin{enumerate}
        \item The functor $\Omega^\infty : \Spt \to \Ani_*$ admits a lift
            \begin{equation*}
                \begin{tikzcd}
                    & \CMon(\Ani) \ar[d] \\
                    \Spt \ar[ur, dashed, "\Omega^\infty"] \ar[r,"\Omega^\infty"'] & \Ani_*
                \end{tikzcd}
            \end{equation*}
        \item The lifted functor $\Omega^\infty$ participates in an adjunction
        \begin{equation*}
            \mathrm{B}^\infty : \CMon(\Ani) \rightleftarrows \Spt : \Omega^\infty
        \end{equation*}
        where the left adjoint $\mathrm{B}^\infty$ is (strong) symmetric monoidal.
        \item The adjunction $\mathrm{B}^\infty \dashv \Omega^\infty$ restricts to a symmetric monoidal equivalence
        \begin{equation*}
            \CMon(\Ani)^\mathrm{gp} \cong \Spt_{\geq 0}
        \end{equation*}
        between grouplike commutative monoids and connective spectra.
    \end{enumerate}
\end{theorem}

In other words, a very large class of spectra can be identified with anima equipped with a commutative monoid structure.
A motivic analogue of the above would be a statement along the lines of ``a suitably large class of motivic spectra can be identified with motivic spaces equipped with some algebraic structure''.
It turns out that the proper replacement for commutative monoids is motivic spaces with framed transfers.

\begin{definition}[See {\cite[Section 2.3 and Section 4]{EHKSY_deloop1}}]
    Let $p : \widetilde{X} \to X$ be a finite syntomic morphism of $S$-schemes.
    A \emph{framing} of $p$ is an equivalence $[\mathrm{L}_p] \cong 0$ in the algebraic $\mathrm{K}$-theory anima $\mathrm{K}(\widetilde{X})$ of $\widetilde{X}$.
    Here, we write $\mathrm{L}_p$ for the cotangent complex of $p$.

    Let $S$ be a scheme.
    We'll write $\Corr^\fr(\Sm_S)$ for the symmetric monoidal $\infty$-category of \emph{framed correspondences}.
    Objects are smooth $S$-schemes, and a morphism from $X$ to $Y$ is a correspondence
    \begin{equation*}
        \begin{tikzcd}[row sep=small]
            & \widetilde{X} \ar[dl,"p"'] \ar[dr,"f"] \\
            X & & Y
        \end{tikzcd}
    \end{equation*}
    of $S$-schemes where $p$ is finite syntomic and equipped with a framing $\alpha : [\mathrm{L}_p] \cong 0$.
    The tensor product in $\Corr^\fr(\Sm_S)$ is the Cartesian product of smooth $S$-schemes.
\end{definition}

\begin{remark}
    Note that a finite syntomic map is finite étale if and only if its cotangent complex is trivial.
    Thus, one might think of framed maps as maps which are potentially worse than finite étale, but in a controlled way.
\end{remark}

\begin{example}
    Recall (\cite[\href{https://stacks.math.columbia.edu/tag/0FKY}{Tag 0FKY}]{stacks-project}) that a morphism $p : \tilde X \to X$ is finite syntomic if and only if, Zariski-locally on the target, it can be written as the composite
    \begin{equation*}
        \Spec A/I \xrightarrow{i} \Spec A \xrightarrow{q} \Spec R
    \end{equation*}
    where $A = R[x_1, \ldots, x_r]$ is a polynomial algebra over $R$, and $I \subseteq A$ is generated by a regular sequence $f_1,\ldots, f_r \in A$.
    Given such a factorization, we can present the cotangent complex $\mathrm{L}_p$ as a two-term complex
    \begin{equation*}
        I/I^2 = \mathscr{N}_i \xrightarrow{\delta} i^* \Omega_q \cong (A/I)\{dx_1, \ldots, dx_r\}
    \end{equation*}
    where $i^* \Omega_q$ is in degree zero, and $\delta$ acts as
    \begin{equation*}
        g + I^2 \mapsto \sum_{i=1}^r \Bigl(\frac{\partial g}{\partial x_i} +I \Bigr) \, dx_i.
    \end{equation*}
    Note that $p$ is finite étale if and only if $\delta$ is an isomorphism.
    In any case, the choice of regular sequence $f_1, \ldots, f_r$ determines a basis for $\mathscr{N}_i$: it is freely generated over $A/I$ by $f_i + I^2$, $i = 1, \ldots, r$.
    Thus, we have equivalences
    \begin{align*}
        [\mathrm{L}_p] &\cong [i^* \Omega_q] - [\mathscr{N}_i] \\
        &\cong r[\mathcal{O}] - r[\mathcal{O}] \\
        &\cong 0
    \end{align*}
    in the algebraic $\mathrm{K}$-theory anima $\mathrm{K}(\Spec A/I)$.
    In summary, the choice of regular sequence $f_1, \ldots, f_r$ determines a framing of $p : \Spec A/I \to \Spec R$.
\end{example}

\begin{remark}
    Note that the morphisms in $\Corr^\fr(\Sm_S)$ are \emph{not} simply correspondences in $\Sm_S$ that satisfy a given property: the framing is \emph{extra data} that one needs to keep track of.
    Consequently, the construction of $\Corr^\fr(\Sm_S)$ is significantly more involved than that of $\Corr^{\fold}(\Sm_S)$ or $\Corr^{\fet}(\Sm_S)$.
    Heuristically, the issue is that one must specify the behavior of the framings under compositions of framed correspondences.
    In \cite[Section 4]{EHKSY_deloop1}, this is taken care of using a bookkeeping device called a \emph{labeling functor} for correspondences.
    See also \cite[Appendix B]{EHKSY_ModMGL}.
\end{remark}

\begin{definition}
    A \emph{motivic space with framed transfers} over $S$ is a presheaf
    \begin{equation*}
        F : \Corr^\fr(\Sm_S)^\op \to \Ani
    \end{equation*}
    whose restriction to $\Sm_S$ is a motivic space as in Definition~\ref{definition:motivic-spaces}.
    We write $\HH^\fr(S) \subseteq \PSh(\Corr^\fr(\Sm_S))$ for the full subcategory spanned by motivic spaces with framed transfers over $S$.
\end{definition}

Again, the inclusion $\HH^\fr(S) \subseteq \PSh(\Corr^\fr(\Sm_S))$ admits a left adjoint $\mathrm{L}_\mot$ called \emph{motivic localization}.
In particular, any presheaf on $\Corr^\fr(\Sm_S)$ gives rise to a motivic space with framed transfers.

Let $F : \Corr^\fr(\Sm_S)^\op \to \Ani$ be a presheaf.
We will usually use the notation
\begin{align*}
    X &\mapsto F(X)\\
    \Bigl(Y \xleftarrow{f} X \Bigr) &\mapsto \Bigl(F(Y) \xrightarrow{f^*} F(X) \Bigr) \\
    \Bigl(\widetilde{X} \xrightarrow{p} X, \quad \alpha : [\mathrm{L}_p] \xrightarrow{\simeq} 0 \Bigr) &\mapsto \Bigl(F (\widetilde{X}) \xrightarrow{(p,\alpha)_*} F(X) \Bigr)
\end{align*}
for the action of $F$ on the objects, forwards morphisms, and backwards morphisms in $\Corr^\fr(\Sm_S)$, respectively.\footnote{
    Note that the scheme $\widetilde{X}$ in the middle of a framed correspondence in $\Corr^\fr(\Sm_S)$ is \emph{not} required to be smooth over $S$.
    Consequently, morphisms in $\Corr^\fr(\Sm_S)$ do not necessarily factor as a composite of a backwards morphism followed by a forwards morphism.
    For sake of exposition, we'll ignore this technical point and continue to describe presheaves on $\Corr^{\fr}(\Sm_S)$ in terms of their actions on forwards morphisms and backwards morphisms.
}
We refer to $(p,\alpha)_*$ as the \emph{transfer} along the framed map $(p,\alpha)$.

\begin{example} \label{example:big-examples}
    For $X \in \Sm_S$, we write
    \begin{itemize}
        \item $\Vect_S(X)$ for the anima of vector bundles on $X$ (\cite[Section 5]{HJNTY_KGL}),
        \item $\FSyn_S(X)$ for the anima of finite syntomic $X$-schemes (\cite[Section 3.3]{EHKSY_ModMGL}),
        \item $\Vect^\sym_S(X)$ for the anima of vector bundles on $X$ equipped with a non-degenerate symmetric bilinear form (\cite[Section 7]{HJNY_KO}), and
        \item $\FSyn^\mathrm{or}_S(X)$ for the anima of finite syntomic $X$-schemes with an orientation, \emph{i.e.} a trivialization of the dualizing sheaf over $X$ (\cite[Section 3.3]{EHKSY_ModMGL}).
    \end{itemize}
    These can be arranged into a diagram
    \begin{equation*}
        \begin{tikzcd}
            \FSyn^\mathrm{or}_S \ar[d] \ar[r] & \Vect^\sym_S \ar[d] \\
            \FSyn_S \ar[r] & \Vect_S
        \end{tikzcd}
    \end{equation*}
    of commutative algebras in $\PSh(\Sm_S)$.
    The commutative algebra structures on $\FSyn_S$ and $\FSyn^\mathrm{or}_S$ are coming from the Cartesian product of finite syntomic schemes, while the commutative algebra structures on $\Vect_S$ and $\Vect^\sym$ are given by tensor product of vector bundles.
    The vertical morphisms are the forgetful maps, while the horizontal maps send a finite syntomic scheme to its structure sheaf.

    We claim that this diagram upgrades to one of commutative algebras in $\PSh(\Corr^\fr(\Sm_S))$.
    Here, we content ourselves with describing the transfers on $\Vect_S$ and $\FSyn^\mathrm{or}_S$, and refer to the above references for further details.
    See also \cite[Section 10]{HJNY_KO}.

    Let $p : \widetilde{X} \to X$ be a finite syntomic map and let $\alpha : [\mathrm{L}_p] \cong 0$ be a framing.
    Given this data, we must specify transfers
    \begin{equation*}
        \Vect_S (\widetilde{X}) \to \Vect_S(X) \qquad\text{and}\qquad \FSyn^\mathrm{or}_S (\widetilde{X}) \to \FSyn^\mathrm{or}_S(X)
    \end{equation*}
    along $(p,\alpha)$.
    On $\Vect_S$, the transfer is given by the pushforward of quasi-coherent sheaves.
    Note that $p$ is finite flat, so this pushforward actually sends vector bundles to vector bundles.
    Note also that this doesn't actually depend on the framing $\alpha$.

    The transfer on $\FSyn^\mathrm{or}_S$ is only slightly more complicated.
    Note that $\mathrm{det}(\mathrm{L}_p) \cong \omega_{\widetilde{X}/X}$.
    (See \cite[Lemma 7.1]{HJNY_KO} for a reference in this generality.)
    Let $q : U \to \widetilde{X}$ be finite syntomic and let $\phi : \omega_{U/X} \cong \mathcal{O}_U$ be a trivialization of the dualizing sheaf.
    The transfer along $(p,\alpha)$ sends this to $U \xrightarrow{q} \widetilde{X} \xrightarrow{p} X$ with the orientation
    \begin{equation*}
        \omega_{U/X} \cong \omega_{U/ \widetilde{X}} \otimes q^* (\omega_{\widetilde{X}/X}) \xrightarrow[\phi \otimes \det(\alpha)]{\simeq} \mathcal{O}_U.
    \end{equation*}
    Note that this depends on the framing $\alpha$, but only via the determinant.
\end{example}

We can now state the recognition principle for motivic infinite loop spaces.

\begin{theorem}[Motivic Recognition Principle, \cite{EHKSY_deloop1}] \label{theorem:recognition}
    Let $S$ be a scheme.
    \begin{enumerate}
        \item The functor $\Omega^\infty_\mathbb{P} : \SH(S) \to \HH(S)_*$ admits a lift
            \begin{equation*}
                \begin{tikzcd}
                    & \HH^\fr(S) \ar[d] \\
                    \SH(S) \ar[ur, dashed, "\Omega^\infty_\mathbb{P}"] \ar[r,"\Omega^\infty_\mathbb{P}"'] & \HH(S)_*
                \end{tikzcd}
            \end{equation*}
        \item The lifted functor $\Omega^\infty_\mathbb{P}$ participates in an adjunction
        \begin{equation*}
            \mathrm{B}^\infty_\mathbb{P} : \HH^\fr(S) \rightleftarrows \SH(S) : \Omega^\infty_\mathbb{P}
        \end{equation*}
        where the left adjoint $\mathrm{B}^\infty_\mathbb{P}$ is (strong) symmetric monoidal.
        \item If $S$ is the spectrum of a perfect field, then the adjunction $\mathrm{B}^\infty_\mathbb{P} \dashv \Omega^\infty_\mathbb{P}$ restricts to a symmetric monoidal equivalence
        \begin{equation*}
            \HH^\fr(S)^\mathrm{gp} \cong \SH(S)^\mathrm{veff}
        \end{equation*}
        between grouplike motivic spaces with framed transfers and very effective motivic spectra.
    \end{enumerate}
\end{theorem}

The proof of Theorem~\ref{theorem:recognition} in \cite{EHKSY_deloop1} is difficult and builds upon the work of Ananyevskiy, Garkusha, Neshitov, and Panin in \cite{GP_FramedMotivesAfterVoevodsky}, \cite{GP_HISheavesFramedTransfers}, \cite{AGP_cancellation}, and \cite{GNP_FramedMotiveRelativeSphere}, which ultimately build upon \cite{Voevodsky_FramedNotes}.

\begin{remark}
    We note that the restriction to perfect fields in Theorem~\ref{theorem:recognition} (3) comes from one of its main ingredients, the cancellation theorem \cite[Theorem~3.5.8]{EHKSY_deloop1} for framed correspondences.
    There, the restriction arises from the use of the theory of strictly $\mathbb{A}^1$-invariant sheaves as developed in \cite{Morel_AlgTop}.
    See \cite{Bachmann_Cancellation} for more on the cancellation theorem, and \cite{Bachmann_Strict} for more on strictly $\mathbb{A}^1$-invariant sheaves.
\end{remark}

For the purposes of the present article, the point of Theorem~\ref{theorem:recognition} is that many motivic spectra of interest admit geometric descriptions via $\mathrm{B}^\infty_\mathbb{P}$. 

\begin{example}
    Let $S$ be a scheme.
    Applying $\mathrm{B}^\infty_\mathbb{P}$ to (the motivic localization of) the diagram from Example~\ref{example:big-examples}, we obtain
    \begin{equation*}
        \begin{tikzcd}
            \mathrm{MSL}_S \ar[d] \ar[r] & \mathrm{ko}_S \ar[d] \\
            \mathrm{MGL}_S \ar[r] & \mathrm{kgl}_S
        \end{tikzcd}
    \end{equation*}
    as a diagram in $\CAlg(\SH(S))$.
    This is a combination of results from \cite{EHKSY_ModMGL}, \cite{HJNTY_KGL}, \cite{HJNTY_KGL}, and \cite{Bachmann_veff}.
    See \cite[Theorem 10.1 and Remark 10.2]{HJNY_KO} for a summary and more specific references.
    (Note that we can drop the assumption that $2 \in \mathcal{O}(S)^\times$ by \cite[Remark~7.3]{HJNY_KO} and \cite{CHN}.)
\end{example}

\begin{remark}
    It is interesting to note that the additive structures and the multiplicative structures that arise naturally in motivic homotopy theory appear to be described using pushforwards along different sorts of maps.
    The gap between framed pushforwards and finite étale pushforwards means that, for now, it is not clear whether there is a nice theory of motivic units and Picard spectra.
\end{remark}

\section{Norms and Transfers, \texorpdfstring{$\mathrm{K}$}{K}-theory and Cobordism} \label{section:end-game}

Comparing the previous section to the one before it, one is led to wonder:

\begin{question}
    Is the motivic recognition principle compatible with norms?
\end{question}

We give an affirmative answer in \cite{Shin_NormsTransfers}.

\begin{theorem}[\cite{Shin_NormsTransfers}]
    The assignment $X \mapsto \Corr^\fr(\Sm_X)$ refines to a norm monoidal $\infty$-category $\Corr^\fr(\Sm_{({-})}) : \Corr^\fet(\Sm_S)^\op \to \Cat_\infty$ over $S$.
\end{theorem}

\begin{remark}
    The construction of $\Corr^\fr(\Sm_{({-})})$ as a norm monoidal $\infty$-category is again rather technical due to the need to keep track of framings as paths in the algebraic $\mathrm{K}$-theory anima.
    In short, the idea can be boiled down to two pieces of input.
    First, norm structures are easy to produce via étale descent (see Slogan~\ref{slogan:norms-by-descent}), which the cotangent complex satisfies.
    Second, the technology of labeling functors is robust enough to overcome the fact that algebraic $\mathrm{K}$-theory itself does not satisfy étale descent.
\end{remark}

With this in hand, the following is almost immediate.

\begin{theorem}[\cite{Shin_NormsTransfers}]\label{theorem:recognition-is-normed}
    Let $S$ be a scheme.
    \begin{enumerate}
        \item The assignment
            \begin{equation*}
                X \mapsto \Bigl(\PSh(\Corr^\fr(\Sm_X)) \xrightarrow{\mathrm{L}_\mot} \HH^\fr(X) \xrightarrow{\mathrm{B}^\infty_\mathbb{P}} \SH(X) \Bigr)
            \end{equation*}
            can be refined to a diagram
            \begin{equation*}
                \PSh(\Corr^\fr(\Sm_{({-})})) \xrightarrow{\mathrm{L}_\mot} \HH^\fr \xrightarrow{\mathrm{B}^\infty_\mathbb{P}} \SH
            \end{equation*}
            of norm monoidal $\infty$-categories over $S$.
            In particular, the functors $\mathrm{L}_\mot$ and $\mathrm{B}^\infty_\mathbb{P}$ both preserve normed algebra structures.
        \item If $S$ is the spectrum of a perfect field, we have an equivalence $\NAlg(\HH^\fr(S)^\mathrm{gp}) \cong \NAlg(\SH(S)^\mathrm{veff})$.
    \end{enumerate}
\end{theorem}

In other words, when $S$ is the spectrum of a perfect field, we can identify certain normed algebras in $\SH(S)$ with certain motivic spaces equipped with finite étale norms and framed transfers.
Note that the perfectness assumption in Theorem~\ref{theorem:recognition-is-normed}~(2) is inherited directly from Theorem~\ref{theorem:recognition}~(3).

Finally, we are ready to prove our main theorem.

\begin{proof}[Proof Sketch of Theorem~\ref{theorem:main}]
    Consider again the diagram
    \begin{equation*}
        \begin{tikzcd}
            \FSyn^\mathrm{or}_S \ar[d] \ar[r] & \Vect^\sym_S \ar[d] \\
            \FSyn_S \ar[r] & \Vect_S
        \end{tikzcd}
    \end{equation*}
    in $\CAlg(\PSh(\Corr^\fr(\Sm_S)))$ from Example~\ref{example:big-examples}.
    Since each presheaf in this diagram satisfies étale descent, Slogan~\ref{slogan:norms-by-descent} tells us that this \emph{automatically} upgrades to a diagram of \emph{normed} algebras in $\PSh(\Corr^\fr(\Sm_S))$.
    By Theorem~\ref{theorem:recognition-is-normed}~(1), the composite $\mathrm{B}^\infty_\mathbb{P} \mathrm{L}_\mot$ preserves normed algebra structures, so we win.
\end{proof}

\begin{remark}
    Note that we have \emph{not} proven that $\mathrm{KO}_S$ itself has a normed algebra structure.
    One might hope that this holds since $\mathrm{KO}_S \cong \mathrm{ko}_S[\tilde{\beta}^{-1}]$ where $\tilde{\beta}$ is a certain element known as the \emph{hermitian Bott element}.
    Unfortunately, it is not the case that inverting elements in a normed algebra results in another normed algebra.
    For instance, one can show that the $\eta$-periodic sphere spectrum $\mathbf{1}_S[\eta^{-1}]$ cannot be a normed algebra, despite the fact that the sphere spectrum $\mathbf{1}_S$ is the initial normed algebra in $\SH(S)$.
    
    To show that $\mathrm{KO}_S$ is also normed, one might hope to use \cite[Proposition 12.6]{BachmannHoyois_Norms}.
    This involves computing norms of the hermitian Bott element.
    So far, it appears that the norms for $\mathrm{ko}_S$ are easiest to compute using $\mathrm{ko}_S \cong \mathrm{B}^\infty_\mathbb{P} \Vect^\sym_S$, but the hermitian Bott element is difficult to access from this perspective.
\end{remark}

\begin{remark}
    An interesting variant of motivic cohomology is \emph{Milnor--Witt motivic cohomology}, which is represented by a motivic spectrum denoted $\mathrm{H}\tilde{\mathbb{Z}}_S$.
    This is a version of motivic cohomology in which algebraic cycles are equipped with symmetric bilinear forms.
    In particular, it provides a natural home for the \emph{Euler class} of vector bundles.
    See \cite{BCDJOe_MWMotives} for more on $\mathrm{H}\tilde{\mathbb{Z}}_S$.

    It is shown in \cite[Section~10]{BEM} that $\mathrm{H}\tilde{\mathbb{Z}}_S \cong \underline{\pi}_0^\mathrm{eff}(\tilde{\mathrm{s}}_0 \mathrm{ko}_S)$, at least when $2 \in \mathcal{O}(S)^\times$.
    See \cite{Bachmann_GeneralizedSlices} for details on this notation.
    In particular, combining our Theorem~\ref{theorem:main} with results from \cite[Section~13]{BachmannHoyois_Norms}, we conclude that $\mathrm{H}\tilde{\mathbb{Z}}_S$ also admits a normed algebra structure, and the canonical map $\mathrm{ko}_S \to \mathrm{H}\tilde{\mathbb{Z}}_S$ is a map of normed algebras.
\end{remark}

\printbibliography

\end{document}